\documentclass[12pt]{article}
\usepackage{amsmath, amsfonts, amsthm, amssymb}
\usepackage{graphicx}
\usepackage{verbatim}

\numberwithin{equation}{section}

\newtheorem{thm}{Theorem}[section]
\newtheorem{lma}[thm]{Lemma}
\newtheorem{cor}[thm]{Corollary}
\newtheorem{defn}[thm]{Definition}

\newtheorem{conj}[thm]{Conjecture}

\renewcommand{\geq}{\geqslant}
\renewcommand{\leq}{\leqslant}

\title{The visible part of plane self-similar sets}
\author{{Kenneth J. Falconer and Jonathan M. Fraser}\\
\small{{\it Mathematical Institute, 
University of St~Andrews, North Haugh, St~Andrews,}} \\
\small{{\it Fife, KY16~9SS, Scotland }}}

\date{}
\begin{document}
\maketitle

\begin{abstract}
Given  a compact subset $F$ of $\mathbb{R}^2$, the visible part $V_\theta F$ of $F$ from direction $\theta$ is the set of $x$ in $F$ such that the half-line from $x$ in direction $\theta$ intersects $F$ only at $x$. It is suggested that if $\dim_H F \geq 1$ then $\dim_H V_\theta F = 1$  for almost all $\theta$, where $\dim_H$ denotes Hausdorff dimension. We confirm this when $F$ is a self-similar set satisfying the convex open set condition and such that the orthogonal projection of $F$ onto every line is an interval. In particular the underlying similarities may involve arbitrary rotations and $F$ need not be connected.

\end{abstract}

\section{Introduction}

The concept of `visibility' has been studied for many years. Nikodym \cite{nik} constructed a remarkable  `linearly accessible set', that is a plane Lebesgue measurable subset $F$ of the unit square of full measure 1 such that for each $x \in F$ there is a straight line $L$ with $F \cap L = \{x\}$, in other words each point of $F$ is visible from two (diametrically opposite) directions.  In the realm of convex geometry,  Krasnosel'skii's theorem provides an elegant criterion for the entire boundary of a compact subset of  $\mathbb{R}^2$ to be visible by direct line of sight from an interior point, see \cite{DBK} for a survey of this area.  More recently, the nature of the visible parts of fractals has been considered, with the question of the Hausdorff dimension of the visible parts attracting particular interest \cite{perc, vis_dim, toby, simon}. 

There are two (related) approaches: given a set $F \subset \mathbb{R}^2$ one can consider the subset of $F$ that is visible from a point $x$ (that is the subset of $F$ that may be joined to $x$ by  a line segment intersecting $F$ at no other points).
Alternatively, the visible part may be considered with respect to a direction $\theta$.  In this case the visible set of $F$ is the subset of  $F$ from which the half-line in direction $\theta$ does not intersect $F$ in any other points. We adopt the latter approach here.

The relationship between the Hausdorff dimension of $F$ and its visible subsets from various points or directions is of particular interest. It has been suggested that if $F$ has Hausdorff dimension greater than 1 then the Hausdorff dimension of the visible subset is 1 from almost all points and from almost all directions. This has been established for certain classes of set, and here we address this question for a class of self-similar subsets of the plane.

For a unit vector $\theta \in \mathbb{R}^2$ we define $l_\theta$ to be the half-line from the origin in direction $\theta$
\[
l_\theta = \{r  \theta: r \geq 0 \}.
\]
Let  $\Pi_\theta = \{x: x\cdot \theta = 0\} $ be the linear subspace perpendicular to $\theta$ and
$\mathrm{proj}_\theta:  \mathbb{R}^2  \to \Pi_\theta $ be orthogonal projection onto $\Pi_\theta$.

\begin{defn}
For a compact $F \subset \mathbb{R}^2$  the \emph{visible part} of $F$ from direction  $\theta$ is
\[
V_\theta F = \big\{ x \in F : (x+l_\theta) \cap F = \{x\}  \big\}.
\]
\end{defn}
\noindent There are natural generalisations of visibility to higher dimensions but we do not discuss these here, see \cite{vis_dim}.  

The Hausdorff dimensions of  projections of sets are summarised in Marstand's projection theorem, see \cite{falconer, marstrand}. We write  $\dim_H$ and  $\dim_B$ to denote Hausdorff and box-counting dimension.

\begin{thm}
Let $F \subset \mathbb{R}^2$ be a Borel set.

(a) If $\dim_H F \leq 1$ then $\dim_H \mathrm{proj}_\theta F = \dim_H F$ for Lebesgue almost all $\theta$,

(b) If $\dim_H F > 1$ then $\dim_H \mathrm{proj}_\theta F = 1$ for Lebesgue almost all $\theta$.
\end{thm}

Since $\dim_H \mathrm{proj}_\theta F \leq  \dim_H  V_\theta F  \leq   \dim_H F$ for all $\theta$,   part (a)  implies that if  $\dim_H F \leq 1$ then $ \dim_H  V_\theta F = \dim_H F$ for Lebesgue almost all $\theta$. If $\dim_H F > 1$ then (b) implies that $ \dim_H  V_\theta F \geq 1 $ for almost all $\theta$ and in many cases there is equality here for almost all $\theta$. Thus one is tempted to make the following conjecture.

\begin{conj} \label{vis conj}
If $F \subset \mathbb{R}^2$ and $\dim_H F \geq1$ then
$\dim_H V_\theta F =1$
for Lebesgue almost all $\theta$.
\end{conj}

This has been verified for several classes of fractal. Simple arguments, for example using rectifiability see \cite{vis_dim},  show that graphs of functions satisfy Conjecture \ref{vis conj} with the only exceptional direction being perpendicular to the $x$-axis. The conjecture holds for quasi-circles, see  \cite{vis_dim}, and this paper also addresses certain connected self-similar sets for which the group generated by the rotational part of the similarity transformations is finite.
In \cite{perc} it is shown that if $F$ is the random set obtained from the fractal percolation process then the conjecture holds almost surely. For an upper dimension bound,   the  part of a compact connected set $F$ that is visible from a  point $x$ has Hausdorff dimension at most
$ \frac{1}{2} +\sqrt{ \dim_H F -\frac{3}{4}}$ for almost all $x \in \mathbb{R}^2$, see \cite{toby}.

Here we give another class of sets for which Conjecture \ref{vis conj} holds.  We show that if a compact set $F$ is self-similar and satisfies the convex open set condition and if $\mathrm{proj}_\theta F$ is an interval for all $\theta$, then the Hausdorff and box-counting dimensions of the visible part $V_\theta F$ equal 1 for all $\theta$.   In particular, we allow the  group generated by the rotational part of the similarity transformations to be infinite and we do not require $F$ to be connected (although the projection condition is automatically satisfied when $F$ is connected).

\section{Main results}
Let  $\{S_i \}_{i=1}^{N}$ be an iterated function system (IFS) consisting of contracting similarity transformations of $\mathbb{R}^2$.  The well-known result of Hutchinson, see \cite{falconer,hutchinson}, states that  there exists a unique non-empty compact set $F$ satisfying
\[
F=\bigcup_{i=1}^N S_i(F).
\]
The set  $F$ is termed the {\em attractor} of the IFS and in our case, where the  $\{S_i \}$ are similarity transformations,  $F$ is termed {\em self-similar}. We write $H$ for the convex hull of $F$.

Recall that the IFS satisfies the {\em open set condition} if there exists a non-empty open set $O$ such that
\begin{equation}
\displaystyle\bigcup_{i=1}^{N} S_i(O) \subseteq O  \label{osc}
\end{equation}
with the union disjoint. (The open set condition guarantees that the Hausdorff dimension of the attractor $F$ equals its similarity dimension, see  \cite{falconer,hutchinson}).
If we can find a {\em convex}  open set satisfying (\ref{osc}) we say that the IFS satisfies the {\em convex open set condition}; it is easy to see that if this is the case we can always take the open set $O$ to be the interior of the convex hull of $F$ (provided $F$ is not a subset of a line), so that 
(\ref{osc}) becomes
\begin{equation}
\displaystyle\bigcup_{i=1}^{N} S_i(\mathrm{int}H) \subseteq \mathrm{int}H \label{cosc}
\end{equation}
with the union disjoint. Indeed, since the inclusion in (\ref{cosc}) is automatic when $H$ is the convex hull of the attractor $F$, the convex open set condition is equivalent to the sets $\{S_i(\mathrm{int}H) \}_{i=1}^{N} $ being disjoint.

\begin{thm}   \label{main}
Let  $\{S_i \}_{i=1}^{N}$ be an IFS of similarity transformations which satisfy the convex open set condition, with attractor $F$. 
Suppose that
$\mathrm{proj}_\theta F$ is an interval for all $\theta$.
Then
\[
\dim_H V_\theta F =\dim_B V_\theta F= 1
\]
for all $\theta$.
\end{thm}

If $F$ is connected then $\mathrm{proj}_\theta F$ is always an interval so  the following corollary is immediate.
\begin{cor}
For every connected self-similar set $F \subset \mathbb{R}$ for which the convex open set condition holds
\[
\dim_H V_\theta F =\dim_B V_\theta F= 1
\]
for all $\theta$.
\end{cor}

If we are only interested in visibility from certain directions, the condition on projection to an interval can be weakened if the group generated by the rotational part of the similarity transformations is finite. Let $G$ be the subgroup of the orthogonal group $O(2)$ generated by the rotational or reflectional components of $\{S_i \}_{i=1}^{N}$ (regarding each similarity $S_i$ as a composition of  a homothety (i.e. a similarity with no rotational component) and a rotation or reflection).

\begin{thm}  \label{finite}
Let  $\{S_i \}_{i=1}^{N}$ be an IFS of similarity transformations, which satisfy the convex open set condition, with attractor $F$. Suppose that the subgroup $G$ is finite. Let $\theta$ be a given direction and suppose that 
$\mathrm{proj}_{g(\theta)} F$ is an interval for all $g \in G$.
Then
\[
\dim_H V_\theta F =\dim_B V_\theta F= 1.\]
\end{thm}

In particular, if each  $S_i$ is a homothety, so that $G$ is the trivial group, we  only need $\mathrm{proj}_\theta F$ to be an interval for the direction of projection under consideration, $\theta$.

\section{Proofs}
We may assume throughout that $F$ is not a subset of a straight line. Write $\{r_i \}_{i=1}^{N}$ for the contraction ratios of $\{S_i \}_{i=1}^{N}$ and let $r_{min}=\min_{i}r_i$ and $r_{max} = \max_{i} r_i$.  We index sets in the construction of $F$ and the points in $F$ in the usual way.  We write $\Sigma^{\infty}$ to denote the sequence space defined by
\[
\Sigma^{\infty}=\{( i_1 ,i_2, i_3 \dots ): 1 \leq i_j \leq N \text{ for all } j\}
\]
and $\Sigma^{k}$ to denote the corresponding sequences of length $k$ 
\[
\Sigma^{k}=\{ (i_1, i_2,\dots , i_k ): 1 \leq i_j \leq N \text{ for all } j =1, \dots , k\},
\]
and
\[
\Sigma= \bigcup_{k \in \mathbb{N}} \Sigma^k
\]
for the set of all finite sequences

For $\underline{i}=(i_1,\dots,i_k) \in \Sigma^k$ we write $H_{\underline{i}} = S_{\underline{i}} (H) \equiv S_{i_1} \circ \dots \circ S_{i_k}(H)$, so that $\{H_{\underline{i}} \}_{\underline{i} \in \Sigma^k}$ are the $k$th-level sets in the usual construction of $F$ from the iterated images of $H$.
Then the map $x:\Sigma^{\infty} \to F$
\[
x( i_1 ,i_2, i_3 \dots )=\bigcap_{k=1}^{\infty} S_{i_1} \circ \cdots \circ S_{i_k} (H)
\]
is surjective but not necessarily bijective.  

\begin{lma} \label{tight}
For all $\underline{i} \in \Sigma^k$ and all directions $\theta $ 
\begin{equation} 
\mathrm{proj}_\theta H_{\underline{i}} = \mathrm{proj}_\theta \displaystyle\bigcup_{j=1}^{N} H_{\underline{i}j}. \label{projs}
\end{equation} 
\end{lma}

\begin{proof} Since $\mathrm{proj}_{\theta}F$ is an interval,
$$
\mathrm{proj}_\theta H=  \mathrm{proj}_{\theta}F 
= \mathrm{proj}_\theta \displaystyle\bigcup_{j=1}^{N} S_j F 
\subseteq  \mathrm{proj}_\theta \displaystyle\bigcup_{j=1}^{N} S_j H  
\subseteq  \mathrm{proj}_\theta H,$$
taking the closure of (\ref{cosc}). 
Taking similar images under $S_{\underline i}$ gives (\ref{projs}).
\end{proof}

In particular, Lemma \ref{tight} means that once a point $x \in F$ is `obscured'  when viewed from direction $\theta$  by some set $H_{\underline{i}}$, it can never become `visible' at a later stage in the construction and hence cannot belong to $V_\theta F$.

We now introduce the notion of $(k,\theta)$-visibility to indicate which $k$th level sets are visible from direction $\theta$.

\begin{defn}
Let $\underline{i} = (i_1,i_2, \dots, i_k) \in \Sigma^k$ and let $\theta$ be a direction.  The $k$th-level set $H_{\underline{i}}$ is \emph{$(k,\theta)$-visible} if there exists a point $x \in H_{\underline{i}}$ such that
\[
(x+ l_\theta) \cap \bigcup_{\underline{j} \in \Sigma^k}H_{\underline{j}}=\{ x\}.
\]
For  $k = 0,1,2,\ldots$  define $V_\theta^k = \{\underline{i} \in \Sigma^k:H_{\underline{i}} \text{ is $(k,\theta)$-visible} \}$.
\end{defn}

The sets  indexed by $V_\theta^k$ provide covers for the  $V_\theta F$ which will eventually give the upper bounds for the dimensions.

\begin{lma} \label{vis}
For all  $\theta$  
\begin{equation}
V_\theta F \subseteq \bigcap_{k=1}^{\infty} \bigcup_{\underline{i} \in V_\theta^k} H_{\underline{i}}. \label{inc}
\end{equation}
\end{lma}

\begin{proof}
For each $k$, if $x \in  \bigcup_{\underline{i} \in \Sigma^k \setminus V_\theta^k} H_{\underline{i}}$ for some $k$, it follows using Definition 3.2 and applying Lemma \ref{tight} repeatedly that $ x \notin V_\theta F$.
\end{proof}
The following lemma relates the visibility of the $H_{\underline{i}}$ to the visibility of their inverse iterates.

\begin{lma}
Let $\theta_1$ be a direction, let $k \in \mathbb{N}$ and let $\underline{i} = (i_1, \dots, i_k) \in \Sigma^k$.  If $H_{\underline{i}}$ is $(k,\theta_1)$-visible then $S_{i_1}^{-1} H_{\underline{i}}$ is $(k-1,\theta_2)$-visible,  where $\theta_2$ is the direction of the half-line $S_{i_1}^{-1} l_{\theta_1}$.
\label{pullback}
\end{lma}

\begin{proof}
Let $\underline{i} = (i_1, \dots, i_k) \in \Sigma^k$ and suppose that $H_{\underline{i}}$ is $(k,\theta_1)$-visible.  By definition there exists a point $x \in H_{\underline{i}}$ such that
\[
(x+ l_{\theta_1}) \cap \bigcup_{\underline{j} \in \Sigma^k}H_{\underline{j}}=\{ x\}.
\]
Suppose
\[
y \in (S_{i_1}^{-1}x+ l_{\theta_2}) \cap \bigcup_{\underline{j} \in \Sigma^{k-1}} H_{\underline{j}}.
\]
Then
\[
S_{i_1} y \in (x+ l_{\theta_1}) \cap \bigcup_{\underline{j} \in \Sigma^k}H_{\underline{j}},
\]
so $S_{i_1} y =x$ and 
\[
y =S_{i_1}^{-1}x \in S_{i_1}^{-1} H_{\underline{i}}
\]
satisfies
\[
(y+ l_{\theta_2}) \cap \bigcup_{\underline{j} \in \Sigma^{k-1}}H_{\underline{j}}=\{ y\};
\]
 hence $S_{i_1}^{-1} H_{\underline{i}}=H_{i_2,\dots,i_k}$ is $(k-1,\theta_2)$-visible.
\end{proof}

For every line $L$ in the plane and $\epsilon >0$ let $L_{\epsilon} = \{x\in \mathbb{R}^2: {\rm dist}(x,L) \leq \epsilon\}$ be the infinite strip centered on $L$ and of width  $2\epsilon$. The following lemma bounds the number of components   $\{ S_i(H)\}$ that overlap such strips.

\begin{lma}\label{strip}
Let the convex set $H$ contain a disc of radius $a_1$ and be contained in a disc of radius $a_2$ and have diameter $|H|$. For all $\epsilon>0$    
let
\begin{equation}
q(\epsilon) = \frac{\lvert H \rvert(2\epsilon+4 a_2 r_{max})}{\pi (a_1 r_{min})^2}. \label{qepdef}
\end{equation}
Then for all lines $L$, the strip $L_{\epsilon}$ intersects at most  $q(\epsilon)$ of the sets $\{\overline{S_i(H)}:1 \leq i \leq N \}$.
\end{lma}

\begin{proof}
Since the $\{S_i\}$ are similarity transformations,  each $\overline{S_i(H)}$ contains a disc of radius $a_1 r_i\geq a_1 r_{min}$ and is contained in a disc of radius $ a_2 r_i \leq a_2 r_{max}$.  If the strip $L_{\epsilon}$ intersects $\overline{S_i(H)}$ then $\overline{S_i(H)} \subseteq L_{\epsilon+2 a_2 r_{max}}$.
Thus if $L_{\epsilon}$ intersects $q$ of the sets $\{\overline{S_i(H)}:1 \leq i \leq N \}$ then, by comparing the areas of $L_{\epsilon}\cap H$ with the areas of the disjoint discs contained in each $\overline{S_i(H)}$, 
\[
q \pi (a_1r_{min})^2 \leq \lvert H \rvert(2\epsilon+4 a_2 r_{max}).
\]
\end{proof}

We now estimate the number of sets $H_{\underline{i}}$ that are $(k,\theta)$-visible from intervals 
$I$ of lengths $|I|$ contained in the linear space $\Pi_\theta$.   For each $k \in \mathbb{N}$, $a>0$ and direction $\theta$, let
\begin{align*}
N_{\theta}(k,a)=\displaystyle\sup_{I\subset \Pi_\theta: \lvert I \rvert = a} \# \Big\{& \underline{i} \in \Sigma^k : \text{there exists } x  \in H_{\underline{i}} \text{ such that } \\
& (x+ l_\theta) \cap \bigcup_{\underline{j} \in \Sigma^k}H_{\underline{j}} = \{x\} \text{ and } \mathrm{proj}_\theta x \in I \Big\}
\end{align*}
and set
\begin{equation}
N(k,a)=\sup_{\theta} N_\theta (k,a) \label{ndef}
\end{equation}

The following lemma gives an upper bound on the growth of $N(k,a)$ with $k$. Although we will ultimately only apply this result in the case $a=\lvert \mathrm{proj}_\theta H \rvert$, the inductive argument requires that we estimate $N(k,a)$ for all $a>0$.

\begin{lma}  \label{keylem}
For all $\epsilon>0$ there exists a constant $c$ such that
\begin{equation}
N(k,a) \equiv \sup_{\theta} N_{\theta}(k,a)\leq c(1 + N \epsilon^{-1} k a)\lambda(\epsilon)^k \label{key}
\end{equation}
 for all $a>0$ and all $k \in \mathbb{N}$, where 
\begin{equation}
\lambda (\epsilon) = \max \{q(\epsilon),r_{\min}^{-1}\}. \label{lambdadef}
\end{equation}
\end{lma}

\begin{proof}
Fix $\epsilon>0$.  We will proceed by induction on $k$.  Note that (\ref{key}) is trivially true when $k=1$ by taking a sufficiently large $c$.  Assume (\ref{key}) for $k-1$ for some $k 
\geq 2$.  We will show that (\ref{key}) holds for $k$ by considering two cases.
\\  \\
Case 1: $a \leq \epsilon$. Fix $\theta$ and let $I \subset \Pi_\theta$ be an interval of length $a\leq \epsilon$. By Lemma \ref{strip} no more than $q(\epsilon)$ first level sets contribute to $N_\theta(1,a)$.  Denote these sets by $S_{i_1}(H), \dots, S_{i_{q(\epsilon)}}(H)$ and suppose that the lengths of the orthogonal projections of the visible part of each of these sets onto $\Pi_\theta$ are $a_{i_1}, \dots, a_{i_{q(\epsilon)}}$  with $a_{i_1}+\dots+a_{i_{q(\epsilon)}}=a$.  Transforming these first level sets back to $H$ and scaling everything accordingly we have, using  Lemma \ref{pullback} and the inductive assumption,
\begin{eqnarray*}
N_\theta (k,a) &\leq& \displaystyle\sum_{j=1}^{q(\epsilon)}\sup_{\theta} N_{\theta}(k-1,a_{i_j}/r_{i_j}) \\
&\leq& \displaystyle\sum_{j=1}^{q(\epsilon)} c\bigg( 1  + N \epsilon^{-1} (k-1) \frac{a_{i_j}}{r_{i_j}}\bigg)\lambda(\epsilon)^{k-1}\\
&\leq& c\Big(q(\epsilon)  
+ N \epsilon^{-1} (k-1)\frac{a}{ r_{min}}\Big)\lambda(\epsilon)^{k-1}\\
&\leq& c\big(1 + N \epsilon^{-1} (k-1) a\big)\lambda(\epsilon)^k \\
&\leq& c(1 + N \epsilon^{-1} k a)\lambda(\epsilon)^k.
\end{eqnarray*}
Taking the supremum over $\theta$ gives (\ref{key}) when $a \leq \epsilon$.
\\ \\
Case 2:
$a > \epsilon$. Fix $\theta$ and let $I \subset \Pi_\theta$ be an interval of length $a > \epsilon$.
 Then no more than $ N$ first level sets contribute to $N_\theta(1,a)$.  Denote these sets by $S_{i_1}(H), \dots, S_{i_s}(H)$ where $s \leq N$ and suppose that the lengths of the orthogonal projections of the visible part of each of these sets onto $\Pi_\theta$ are $a_{i_1}, \dots ,a_{i_s}$ with $a_{i_1}+\dots+a_{i_s}=a$.  Transforming these first level sets back to $H$ and scaling everything accordingly we have, using  Lemma \ref{pullback} and the inductive assumption,
\begin{eqnarray*}
N_\theta (k,a) &\leq& \displaystyle\sum_{j=1}^{s}\sup_{\theta} N_{\theta}(k-1,a_{i_j}/r_{i_j}) \\
&\leq& \displaystyle\sum_{j=1}^{s} c\bigg( 1  + N \epsilon^{-1} (k-1) \frac{a_{i_j}}{r_{i_j}}\bigg)\lambda(\epsilon)^{k-1}\\
&\leq& c\Big(s  
+ N \epsilon^{-1} (k-1)\frac{a}{ r_{min}}\Big)\lambda(\epsilon)^{k-1}\\
&\leq& c\big(N\epsilon^{-1}a + N \epsilon^{-1} (k-1) a\big)\lambda(\epsilon)^k \qquad\text{since $1<\epsilon^{-1}a$} \\ 
&\leq& c(1 + N \epsilon^{-1} k a)\lambda(\epsilon)^k.
\end{eqnarray*}
Taking the supremum over $\theta$ gives (\ref{key}) when $a > \epsilon$.
\end{proof}

Thus the asymptotic growth of $N(k,a)$ is at most of order $\lambda(\epsilon)^k$. This enables us to obtain an upper bound for the dimension of the visible sets.

\begin{lma}  Given $\epsilon>0$, for all $\theta$,
\[
\dim_H F \leq \overline{\dim}_B V_\theta  F \leq \frac{\log \lambda(\epsilon)}{-\log r_{\max}}.
\]
\label{box}
\end{lma} 

\begin{proof}
  By Lemma \ref{vis},  for all $k \in \mathbb{N}$ we can find a cover of $V_\theta F$ by no more than $\#V_\theta^k=N_\theta(k,\lvert\mathrm{proj}_\theta H \rvert) \leq N(k, \lvert\mathrm{proj}_\theta H \rvert)$ sets of diameter at most $\lvert H \rvert r_{\max}^k$.  Using  Lemma \ref{keylem} and the definition of upper box dimension,

\begin{eqnarray*}
\overline{\dim}_B V_\theta F &\leq& \limsup_{k \to \infty} \frac{\log N(k, \lvert\mathrm{proj}_\theta H \rvert)}{-\log \lvert H \rvert r_{\max}^k }\\ \\
&\leq& \limsup_{k \to \infty} \frac{\log \Big(c(1 + N \epsilon^{-1} k a)\lambda(\epsilon)^k \Big)}{-\log  \lvert H \rvert r_{\max}^k } \\ \\
&=&\frac{\log \lambda(\epsilon)}{-\log r_{\max}}.
\end{eqnarray*}
\end{proof}

We complete the proof of  Theorem \ref{main} by applying Lemma \ref{box} to  $F$ regarded as the attractor of an alternative IFS $\{ S_{\underline{i}}: \underline{i} \in T_\delta\}$ where $T_\delta$ is a suitable `stopping'.
\\ \\
\noindent {\em Proof of Theorem \ref{main}}.
For $\delta>0$ define the {\em stopping} $T_\delta$ by
\[
T_\delta= \big\{\underline{i}=i_1 \dots i_k \in \Sigma: r_{i_1} \dots r_{i_k} < \delta \leq  r_{i_1} \dots r_{i_{k-1}} \big\}.
\]
Then the IFS $\{S_{\underline{i}} \}_{\underline{i} \in T_\delta}$ has $F$ as its attractor.  Taking $\epsilon=\delta$, and writing $r_{\underline{i}} = r_{i_1} \dots r_{i_k}$ for ${\underline{i}}= (i_1,\ldots,i_k)$,  Lemma \ref{strip} gives 
\begin{equation}
q (\delta) = \frac{\lvert H \rvert(2\delta+4 a_2\max_{\underline{i} \in T_\delta} r_{\underline{i}})}{\pi (a_1\min_{\underline{i} \in T_\delta} r_{\underline{i}})^2} \leq \frac{\lvert H \rvert(2\delta+4a_2 \delta)}{\pi (a_1\delta r_{min})^2} \equiv \delta^{-1} K, \label{kdef}
\end{equation}
for some constant $K$ independent of $\delta$. 
Applying Lemma \ref{box}, (\ref{lambdadef}) and (\ref{kdef})  to the IFS $\{S_{\underline{i}} \}_{\underline{i} \in T_\delta}$  we obtain
\begin{eqnarray*}
1 \leq \dim_H V_\theta F \leq \overline{\dim}_B V_\theta F &\leq&  \frac{\log \lambda(\delta)}{-\log (\max_{\underline{i} \in T_\delta}r_{\underline{i}})}\\ \\
&\leq& \max  \bigg\{\frac{\log q(\delta)}{-\log (\max_{\underline{i} \in T_\delta}r_{\underline{i}})},\frac{\log  (\min_{\underline{i} \in T_\delta}r_{\underline{i}})^{-1}}{-\log  (\max_{\underline{i} \in T_\delta}r_{\underline{i}})}\bigg\}\\ \\
&\leq& \max  \bigg\{\frac{\log \delta^{-1} K}{-\log \delta},\frac{\log \delta^{-1} r_{min}^{-1}}{-\log \delta}\bigg\}\\ \\
&=& \max  \bigg\{1-\frac{\log  K}{\log \delta}, ~~1+\frac{\log r_{min}}{ \log \delta } \bigg\}.
\end{eqnarray*}
This is true for all $\delta > 0$, so
\[
\dim_H V_\theta F=\dim_B V_\theta= 1.
\]
\begin{flushright}$\Box$\end{flushright}

The proof of Theorem \ref{finite} is very similar to the proof of Theorem \ref{main}.  The difference is found in the definition of $N(k,a)$ as instead of taking the supremum over all directions we need only consider the set of directions $\{g(\theta):   g \in G\}$.  Hence, we replace (\ref{ndef}) by
$$N(k,a)=\max_{g \in G} N_{g(\theta)}(k,a) \label{new}$$
for each $k \in \mathbb{N}$ and $a>0$, and the proof proceeds in the same way.

\section{Examples}

A large range of self-similar sets satisfy the convex open set condition and project onto intervals in all directions and so satisfy the conditions of Theorem \ref{main}. 

Let $H$ be a convex set and $\{S_i \}_{i=1}^{N}$ be similarities such that $H$ is the convex hull of $\bigcup_{i=1}^{N} S_i(H)$ and such that every straight line that intersects $H$ intersects at least one of the 
$S_i(H)$ (such sets need not be connected). Then the hypotheses and conclusion of Theorem \ref{main} hold. In particular if $F$ is a self-similar carpet, so that $H$ is a unit square divided into $n \times n$ subsquares of side $1/n$ where $n \geq 2$ and the $S_i(H)$ are a subcollection of these subsquares such that every line that intersects $H$ intersects at least one of the $S_i(H)$ (the $S_i(H)$ will include the four corner squares), then we can conclude that the visible sets from all directions have dimension 1.

\begin{figure}
	\centering
	\includegraphics[width=50mm]{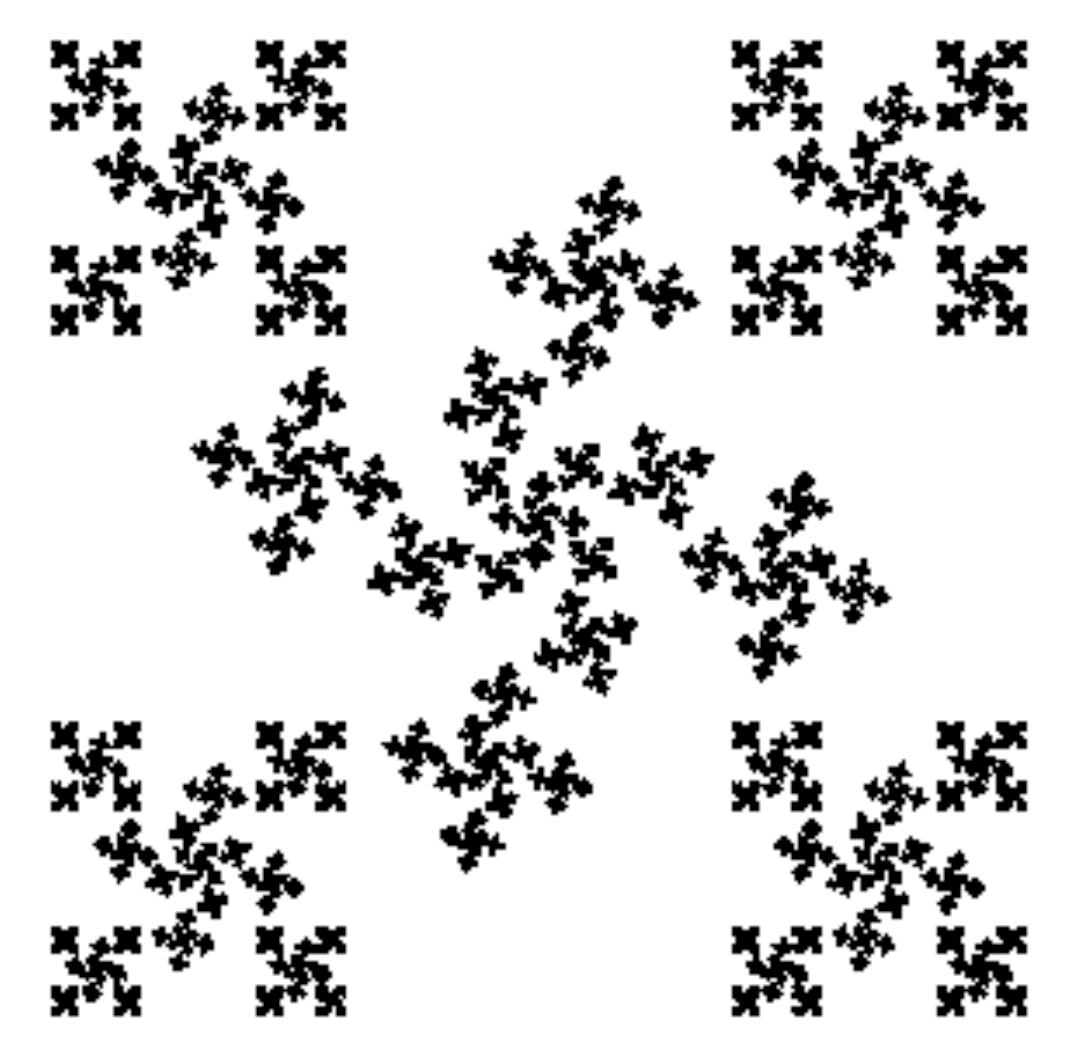}
	\qquad
	\includegraphics[width=48mm]{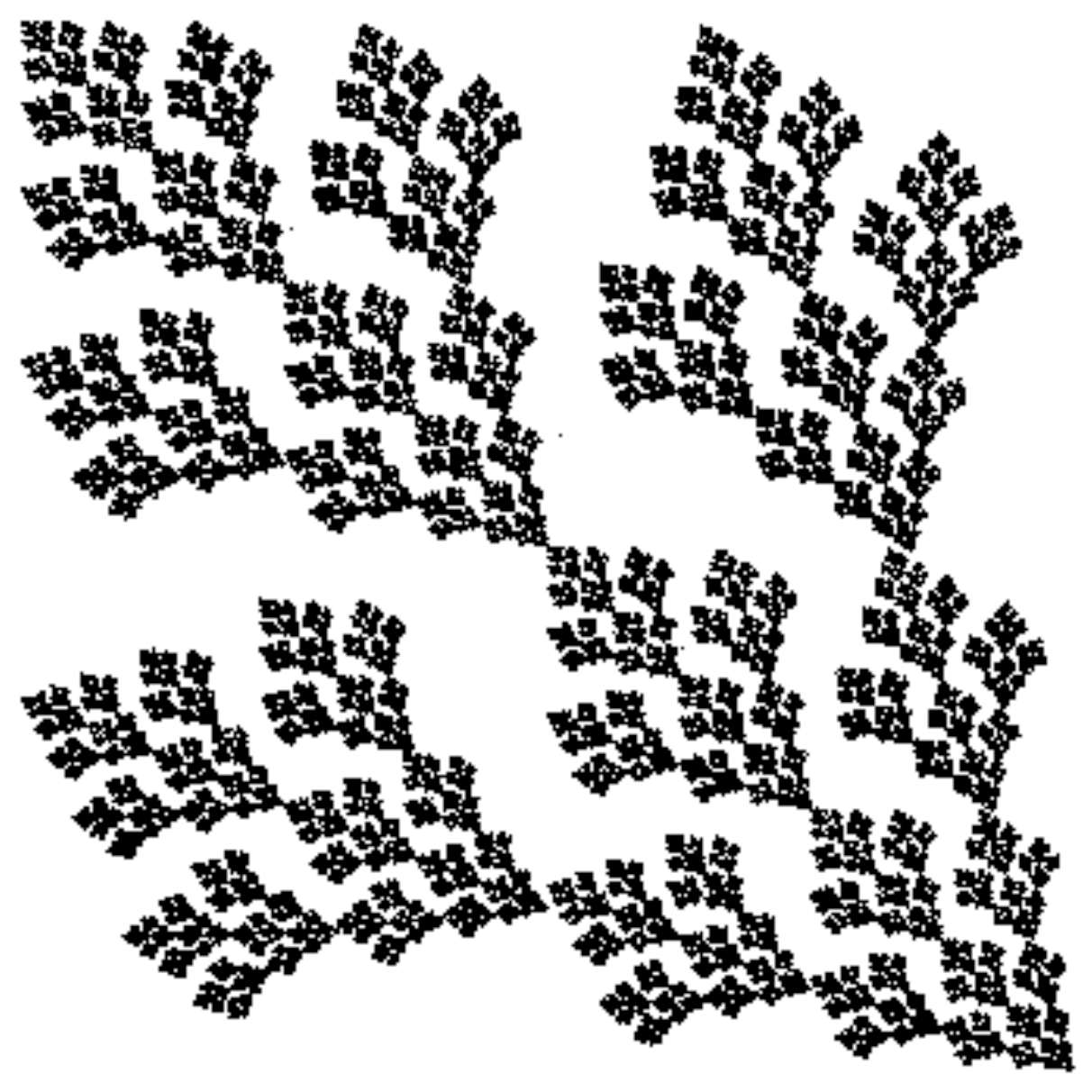}
	\caption{Two self-similar sets with visible sets from all directions of dimension 1}
	\label{figureA.pdf}
\end{figure}

Many self-similar fractals constructed by the `initiator-generator' procedure (see \cite{falconer,mand}) satisfy the conditions of Theorem \ref{main} (this procedure generalises the usual von Koch curve construction by repeated substitution of an open  polygon in itself). The generator is an open polygon $K$ consisting of the union of a finite number of line segments. A self-similar curve is constructed by repeatedly replacing line segments by similar copies of $K$ scaled so that the endpoints of $K$ are mapped to the ends of each of the line segments.  Thus the generator codes a family of similarity transformations $\{S_i \}$ that map the ends of $K$ onto its component line segments. Theorem \ref{main} applies provided that the convex hull of $K$ is mapped into itself by these similarities with the images of the interiors disjoint.
Particular instances include the generalised von Koch curves, where the generator consists of four equal  line segments with angles $-(\frac{\pi}{2} + \alpha), 2\alpha, \frac{\pi}{2} + \alpha$ between consecutive segments, for some $0<\alpha < \frac{\pi}{2}$ ($\alpha = \frac{\pi}{6} $ gives the usual von Koch curve). 
Topologically more complicated examples can easily be obtained using generators other than curves.

\end{document}